\documentclass{amsart}
\usepackage{comment}
\usepackage{amscd}
\usepackage{amssymb}
\usepackage[all, knot]{xy}
\usepackage[top=1.2in, bottom=1.2in, left=1.2in, right=1.2in]{geometry}
\xyoption{all}
\xyoption{arc}
\usepackage{hyperref}

\def\HomLF{\operatorname{{{Hom}}_{\LF}}}

\def\Aut{\operatorname{Aut}}

\def\s{\sigma}

\def\del{\partial}

\def\cube#1#2#3#4#5#6#7#8{
& #5 \ar[rr] \ar[dl] \ar@{-}[d] && #6 \ar[dd] \ar[dl] \\
#1 \ar[rr] \ar[dd]  & \ar[d] & #2 \ar[dd] \\
& #7 \ar@{-}[r] \ar[dl] & \ar[r] & #8 \ar[dl] \\
#3 \ar[rr] && #4 \\
}

\def\Hom{\operatorname{Hom}}

\def\a{\alpha}
\def\b{\beta}

\def\Tot{\operatorname{Tot}}

\def\sing{{\text{sing}}}

\def\image{\operatorname{im}}
\def\im{\image}
\def\ker{\operatorname{ker}}

\def\cP{\mathcal P}

\def\cD{\mathcal D}

\def\coker{\operatorname{coker}}

\def\dm{\operatorname{dim}}

\def\Sing{\operatorname{Sing}}

\def\Tor{\operatorname{Tor}}

\def\Spec{\operatorname{Spec}}

\def\lra{\longrightarrow}
\def\into{\hookrightarrow}
\def\onto{\twoheadrightarrow}
\def\cH{\mathcal{H}}

\def\e{\epsilon}

\newcommand{\Q}{\mathbb{Q}}

\newcommand{\C}{\mathbb{C}}
\newcommand{\Z}{\mathbb{Z}}
\newcommand{\N}{\mathbb{N}}

\newcommand{\fp}{{\mathfrak p}}
\newcommand{\fm}{{\mathfrak m}}

\numberwithin{equation}{section}

\theoremstyle{plain} %% This is the default, anyway
\newtheorem{thm}[equation]{Theorem}

\newtheorem{introthm}{Theorem}
\newtheorem{introcor}[introthm]{Corollary}
\newtheorem*{introthm*}{Theorem}

\newtheorem{cor}[equation]{Corollary}
\newtheorem{lem}[equation]{Lemma}
\newtheorem{prop}[equation]{Proposition}

\newtheorem{quest}[equation]{Question}

\theoremstyle{definition}
\newtheorem{defn}[equation]{Definition}

\newtheorem{ex}[equation]{Example}

\theoremstyle{remark}
\newtheorem{rem}[equation]{Remark}

\def\Dsing{D_{\operatorname{sing}}}

\def\Perf{\operatorname{Perf}}

\newcommand{\supp}{\operatorname{supp}}

\newcommand{\xra}[1]{\xrightarrow{#1}}
\newcommand{\xla}[1]{\xleftarrow{#1}}

\newcommand{\id}{\operatorname{id}}

\def\lf{\operatorname{lf}}
\def\LF{\operatorname{LF}}
\def\mf{\operatorname{mf}}
\def\lffl{\operatorname{lf}^{\mathrm{fl}}}
\def\mffl{\operatorname{mf}^{\mathrm{fl}}}
\def\len{\operatorname{length}}
\def\Hlf{\operatorname{H}}

\def\codim{\operatorname{codim}}

\def\cPsi{\psi_{\mathrm{cyc}}}
\def\Fold{\operatorname{Fold}}

\def\mult{\operatorname{mult}}
\def\stable{\mathrm{stable}}
\def\and{{ \text{ and } }}

\def\res{\operatorname{res}}

\def\tB{\tilde{B}}
\def\tchi{\tilde{\chi}}

\def\bt{{\mathbf t}}

\newcommand{\cyc}{\operatorname{cyc}}

\begin{document}
\begin{abstract}
We define Adams operations on matrix factorizations, and we show these operations enjoy analogues of several key properties of the Adams operations on perfect complexes with support developed by Gillet-Soul\'e in \cite{GS87}. As an application, we give a proof of a conjecture of Dao-Kurano concerning the vanishing of Hochster's $\theta$ invariant.
\end{abstract}

\title{Adams Operations on Matrix Factorizations}

\author{Michael K. Brown}
\address{Hausdorff Center for Mathematics, Villa Maria, Endenicher Allee 62, D-53115 Bonn, Germany}
\email{mbrown@math.uni-bonn.de}

\author{Claudia Miller}
\address{Mathematics Department, Syracuse University, Syracuse, NY 13244-1150, USA}
\email{clamille@syr.edu}

\author{Peder Thompson}
\address{Department of Mathematics and Statistics, Texas Tech University, Broadway and Boston, Lubbock, TX 79409}
\email{peder.thompson@ttu.edu}

\author{Mark E. Walker}
\address{Department of Mathematics, University of Nebraska, Lincoln, NE 68588-0130, USA}
\email{mark.walker@unl.edu}

\thanks{This work was partially supported by a grant from the
 Simons Foundation 
(\#318705 for Mark Walker) 
and grants from the National Science Foundation 
(NSF Award DMS-0838463 for Michael Brown and Peder Thompson and 
NSF Award DMS-1003384 for Claudia Miller).}

\maketitle
\tableofcontents

%%%%%%%%%%%%%%%%%%%%%%%
%%%%%%%%%%%%%%%%%%%%%%%
\section{Introduction}
%%%%%%%%%%%%%%%%%%%%%%%
%%%%%%%%%%%%%%%%%%%%%%%
The goal of this paper is to establish a theory of Adams operations on the Grothendieck group of matrix factorizations and to use these
operations to prove a conjecture of Dao-Kurano (\cite{ DK12} Conjecture 3.1 (2)) concerning the vanishing of Hochster's $\theta$ pairing for a pair of modules defined on an isolated hypersurface
singularity.

Let $Q$ be a commutative Noetherian ring, and let $f \in Q$. A {\em matrix factorization} of $f$ in $Q$ is a 
$\Z/2$-graded, finitely generated projective $Q$-module $P = P_0 \oplus P_1$, equipped with an odd degree $Q$-linear endomorphism $d$ satisfying
$d^2 = f \id_P$. In other words, a matrix factorization is a pair of maps of finitely generated projective $Q$-modules, $(\a: P_1 \to P_0, \b:  P_0 \to P_1)$, satisfying $\a
\b = f \id_{P_0}$ and $\b \a = f \id_{P_1}$. 

When $f = 0$, a matrix factorization of $f$ is the same thing as a $\Z/2$-graded complex of finitely generated projective $Q$-modules. In this case, we have the
evident $\Z/2$-graded analogues of chain maps and homotopies of such. These, in fact, generalize to an arbitrary $f$: 
the matrix factorizations of $f \in Q$ form the objects of a category $\mf(Q,f)$, in which 
a morphism between objects $P$ and $P'$ of $\mf(Q,f)$ is a degree zero $Q$-linear map $g: P \to P'$ such that $d_{P'} \circ g = g \circ d_P$. In other words, a
morphism is a pair of maps $g_0: P_0 \to P_0'$ and $g_1: P_1 \to P_1'$ causing the evident pair of squares to commute.
A {\em homotopy} joining morphisms $g_1, g_2: P \to P'$ in $\mf(Q,f)$ is 
a $Q$-linear map $h: P \to P'$ of odd degree such that $d_{P'} h + h d_P = g_1
- g_2$. The {\em homotopy category} of $\mf(Q,f)$ is the category $[\mf(Q,f)]$ obtained from $\mf(Q,f)$ by identifying homotopic morphisms. It is well-known
that, when $Q$ is regular 
and $f$ is a non-zero-divisor, $[\mf(Q,f)]$ may be equipped with a canonical triangulated structure (see,
for instance, \cite{orlov2003triangulatedInRussian} Section 3.1). 

Much of the interest in matrix factorizations arises from the following result. For a Noetherian ring $R$, let
$D^b(R)$ denote the bounded derived category of $R$. Objects of $D^b(R)$ are 
bounded complexes of finitely generated $R$-modules, and morphisms are obtained from chain maps by inverting the collection of quasi-isomorphisms. 
Let $\Perf(R)$ denote the full triangulated subcategory of $D^b(R)$ consisting of bounded complexes of
finitely generated and projective $R$-modules, and let $\Dsing(R)$ denote the Verdier quotient $D^b(R)/\Perf(R)$, called the \emph{singularity category} of $R$. 
The following theorem is essentially due to work of Buchweitz and Eisenbud  in \cite{buchweitz1986maximal,eisenbud1980homological}; this particular
formulation of the result is proven by Orlov in \cite{orlov2003triangulatedInRussian}: 

\begin{introthm}\cite[Theorem 3.9]{orlov2003triangulatedInRussian} \label{introthm1}
If $Q$ is regular 
and $f$ is a non-zero-divisor, there is an equivalence of triangulated categories
$$
[\mf(Q,f)] \xrightarrow{\cong}
\Dsing(Q/(f))
$$
determined by sending a matrix factorization $(\a: P_1 \to P_0, \b:
P_0 \to P_1)$ to $\coker(\a)$.
\end{introthm}

\begin{rem} In \cite {orlov2003triangulatedInRussian}, Orlov assumes $Q$ contains a field and has finite Krull dimension, but these assumptions are in fact not needed
  for this Theorem to hold.
\end{rem} 

Let $R:=Q/(f)$. Under the assumptions of Theorem~\ref{introthm1}, the Grothendieck group $K_0(\mf(Q,f))$ of the triangulated category $[\mf(Q,f)]$ is isomorphic to the quotient $\frac{G_0(R)}{\im(K_0(R) \to G_0(R))}$. So, 
defining a notion of Adams operations on $K_0(\mf(Q,f))$, in this setting, amounts to defining such operations on this quotient.

For a closed subset $Z$ of $\Spec(Q)$, define $\cP^Z(Q)$ to be
the category of bounded complexes of finitely generated and projective $Q$-modules whose homology is supported on $Z$.
Gillet-Soul\'e define lambda and Adams operations on the
Grothendieck group $K_0^Z(Q) := K_0(\cP^Z(Q))$ (\cite{GS87} Sections 3 and 4). It is tempting to mimic their approach to define Adams operations on $K_0(\mf(Q,f))$, since $\mf(Q,f)$ is somewhat analogous to $\cP^{V(f)}(Q)$. But their construction relies on the Dold-Kan correspondence relating $\N$-graded complexes to simplicial modules; since matrix factorizations are $\Z/2$-graded, such an approach is not available for 
$K_0(\mf(Q,f))$.

Instead, we model our approach after the construction of the \emph{cyclic Adams operations} $\psi^p_{\cyc}$ on $K_0^Z(Q)$ developed in
\cite{brown2016cyclic} (see also \cite{Ati66}, \cite{Haution}, and \cite{Kock}). 
Let us give a brief summary of the construction of the operations $\cPsi^p$ and some of their properties.

Fix a prime $p$. We assume that $p$ is invertible in $Q$ and that $Q$ contains all $p$-th roots of unity (when $Q$ is local, the
case of primary interest to us, we can find such a prime $p$, at least after passing to a faithfully flat extension of $Q$).
For a perfect complex of $Q$-modules $X$, let
$T^p(X)$ denote the $p$-th tensor power of $X$, which comes equipped with a canonical left action by the symmetric group $\Sigma_p$. 
For a $p$-th root of unity $w \in Q$, set 
$T^p(X)^{(w)}$ to be the eigenspace of eigenvalue $w$ for the action of the $p$-cycle $(1 \, 2 \, \cdots \, p)$ on $T^p(X)$. 
We define
$$
\cPsi^p(X) = [T^p(X)^{(1)}] - [T^p(X)^{(\zeta)}]
$$
where $\zeta$ is a primitive $p$-th root of unity.

In Sections 2 and 3 of \cite{brown2016cyclic}, it is established that this formula induces a
well-defined operation on $K_0^Z(Q)$ (see also \cite{Haution}). 
In fact, by Corollary 6.14 of loc. cit., if $p!$ is invertible in $Q$, $\cPsi^p$ agrees with the $p$-th Adams operation on $K_0^Z(Q)$ defined by Gillet-Soul\'e. More generally, we have:

\begin{introthm}\cite[Theorem 3.7]{brown2016cyclic} \label{introthm2} 
If $p$ is a prime, and $Q$ contains $\frac{1}{p}$ and all the $p$-th roots of unity, 
then the action of $\cPsi^p$ on
  $K_0^Z(Q)$ satisfies the four Gillet-Soul\'e axioms 
  defining a degree $p$ Adams operation. 
\end{introthm}

We refer the reader to Theorem 3.7 of \cite{brown2016cyclic} for a precise statement of the four Gillet-Soul\'e axioms. A consequence of Theorem~\ref{introthm2} is that the action of $\cPsi^p$ on $K_0^Z(Q)_\Q:=K_0^Z(Q) \otimes \Q$ is diagonalizable: 
there is a ``weight decomposition''
$$
K_0^Z(Q)_\Q = \bigoplus_{i =c}^d  K_0^Z(Q)_{\Q}^{(i)},
$$
where $K_0^Z(Q)_{\Q}^{(i)}$ is the eigenspace of $\cPsi^p$ of eigenvalue $p^i$, and $c$ is the codimension of $Z$ (loc. cit. Corollary 3.12). 

In Section~\ref{operations}, we use the operations $\cPsi^p$ as a model to construct cyclic Adams operations $\cPsi^p$ on the Grothendieck group
$K_0(\mf(Q,f))$, as well as more general versions for matrix factorizations with a support condition.
In Theorem \ref{psi-GS-axioms-mf} and Proposition \ref{prop727},  we prove:

\begin{introthm} \label{introthm3} If $p$ is prime, and $Q$ contains $\frac{1}{p}$ and all the $p$-th roots of unity, the operator $\cPsi^p$ on $K_0(\mf(Q,f))$
satisfies the evident analogues of the four Gillet-Soul\'e axioms for a $p$-th Adams operation. 

Moreover, if $Q$ is regular 
and $f \in Q$ is a non-zero-divisor, the canonical surjection
$$
K_0^{V(f)}(Q)\onto K_0(\mf(Q,f))
$$
is compatible with the action of $\cPsi^p$.
\end{introthm}

For $Q$ regular, $f$ a non-zero-divisor, and $R = Q/(f)$, 
given a finitely generated $R$-module $M$, let $[M]_\stable \in K_0(\mf(Q,f))$ denote the image of $[M] \in G_0(R)$ under the canonical surjection
$G_0(R) \onto K_0(\mf(Q,f))$ given by Theorem \ref{introthm1}.

\begin{introcor}
\label{introcor} 
Assume $Q$ is a regular ring 
containing $\frac{1}{p}$ and all the $p$-th roots of unity for some
prime $p$, and suppose $f \in Q$ is a non-zero-divisor. The action of $\cPsi^p$ induces an eigenspace decomposition
$$
K_0(\mf(Q,f))_\Q
= 
\bigoplus_{i =1}^d  K_0(\mf(Q,f))_\Q^{(i)}
$$
Moreover, if $M$ is a finitely generated $R$-module, then
$$
[M]_\stable \ 
\in 
\bigoplus_{i=\codim_{R} M +1}^d  K_0(\mf(Q,f))_\Q^{(i)}.
$$
\end{introcor}

In Section~\ref{DK}, we give an application of the above results. For
the rest of this introduction, assume $Q$ is a regular local ring with
maximal ideal $\fm$, and assume $f$ is a non-zero element of
$\fm$. Assume also that $R = Q/(f)$ is an
isolated singularity: that is, $R_\fp$ is regular for all $\fp \in \Spec(R) \setminus \{\fm\}$.   
Then for any pair of finitely generated $R$-modules $(M,N)$, we have 
$$
\Tor^R_i(M,N) \cong \Tor^R_{i+2}(M,N)
$$ 
and
$$
\len \Tor^R_i(M,N) < \infty
$$
for $i \gg 0$. This motivates the following definition.

\begin{defn} \label{def728}
With $Q, f, R$ as above, for a pair of finitely generated $R$-modules $(M, N)$,
set
$$
\theta_R(M,N) = \len\left(\Tor_{2i}^R(M,N)\right) -  \len\left(\Tor_{2i+1}^R(M,N)\right)
$$
for $i \gg 0$. 
\end{defn}

The pairing $\theta_R(-,-)$ is called \emph{Hochster's theta pairing}, since it first appeared in work of Hochster \cite{hochster1981dimension}.
The theta pairing should be
regarded as the analogue, for the singularity category $D_\sing(R)$, of the intersection multiplicity pairing that occurs, for example, in Serre's multiplicity conjectures. 
There has been much recent work on better understanding the theta pairing, including when it vanishes and how it relates to more classical
invariants. 
Buchweitz and van Straten \cite{BvS} show that, for complex isolated hypersurface singularities,  the theta pairing can be recovered from 
the linking form on the link of an isolated singularity. In the same setting, Polishchuk and Vaintrob \cite{PV} relate it to the classical residue pairing using
the boundary bulk map. 
It was conjectured by Dao that $\theta$ vanishes for all isolated hypersurface singularities $R$ such that
$\dm(R)$ is even, and this has now been proven in almost all cases; see \cite{PV, BvS, MPSW, WalkerTheta}.
We refer the reader to Section 3 of \cite{DK12} for additional history of the theta pairing
and a list of several other conjectures. 

One such conjecture, loc.\ cit.\ Conjecture 3.1 (2), is an analogue
of Serre's Vanishing Conjecture (cf.\ the Remark on page 111 of \cite{Ser65}).  
This conjecture was proven by Dao in the case where $R$ is excellent and contains field, using a geometric approach (\cite{dao2013decent} Theorem 3.5).
As an application of the properties of Adams operations on matrix factorizations that we establish in Section~\ref{operations}, we prove this conjecture in full generality:

\begin{introthm}[cf. Theorem \ref{dao}]\label{introthm5}
Let $(Q, \fm)$ be a regular local ring and  $f \in \fm$ with $f \neq 0$. Suppose that $R=Q/(f)$ is an isolated singularity. If $M$ and $N$ are finitely generated $R$-modules such that 
$$
\dim M + \dim N \leq \dim R
$$
then $\theta_R(M,N)=0$.   
\end{introthm}

We close this introduction with a sketch of our proof of Theorem~\ref{introthm5}. 
We easily reduce to the case where there is a prime $p$ such that $Q$ contains $\frac{1}{p}$ and all $p$-th roots of unity. Given a matrix factorization $P=(\alpha: P_1 \to P_0, \beta: P_0 \to P_1)$ of $f$, one may obtain a matrix factorization $P^\circ$ of $-f$ by negating $\beta$. 
In Proposition~\ref{prop727b}, we show 
$$
\theta_R(M,N) = \chi \left( [M]_\stable \cup [N]_\stable^\circ \right),
$$
where $- \cup -$ is the pairing  induced by tensor product of matrix factorizations,
and $\chi$ denotes the Euler characteristic. 
The assumptions ensure that $[M]_\stable \cup [N]_\stable^\circ$ is a class in $K_0(\mf^{\fm}(Q,0))$, the Grothendieck group of $\Z/2$-graded complexes of
finitely generated projective $Q$-modules with finite length homology, 
so that $\chi$ is well-defined. By Corollary~\ref{introcor} and the linearity of $\chi$, we may assume that the classes $[M]_\stable, [N]_\stable$ lie in
eigenspaces 
$K_0(\mf(Q, 0))_\Q^{(i)}$ and $K_0(\mf(Q, 0))_\Q^{(j)}$, respectively, where $i + j > d = \dim Q$.
By properties of the operations $\cPsi^p$ established in Theorem~\ref{introthm3}, 
$[M]_\stable \cup [N]_\stable^\circ \in K_0(\mf^{\fm}(Q, 0))_\Q^{(i+j)}$.

At this point, one would like to argue that $K_0(\mf^{\fm}(Q,0))_\Q = K_0(\mf^{\fm}(Q,0))^{(d)}_\Q$, which would force $[M]_\stable \cup [N]_\stable^\circ  = 0$. 
Indeed, one might expect that $K_0(\mf^{\fm}(Q,0))$ is generated by the $\Z/2$-folding of the class of the Koszul complex on a regular sequence of
generators of $\fm$, which lies in $K_0(\mf^{\fm}(Q,0))^{(d)}$ by the axioms in Theorem~\ref{introthm3}; 
this would be parallel to what occurs for bounded $\Z$-graded complexes. The proof of Theorem~\ref{introthm5} sketched here would then be almost exactly the same  as 
Gillet and Soul\'e's proof of Serre's Vanishing Conjecture. 

We are not able to prove $K_0(\mf^{\fm}(Q,0))$ is generated by the Koszul complex, and indeed we have come to suspect this might be false (see Example~\ref{koszul}). Fortunately, for
the proof of Dao-Kurano's conjecture, one needs only the weaker property that there is an equality of maps $\chi \circ \cPsi^p = p^d \chi$ from $K_0(\mf^{\fm}(Q,0))$ to $\Z$; we prove this in Theorem \ref{KeyLemma}.

We thank Luchezar Avramov for helpful conversations in preparing this paper, and we thank 
Dave Benson, Oliver Haution, Bernhard K\"ock,   and Paul Roberts 
for leading us to the relevant papers 
\cite{Ben84}, \cite{Haution}, \cite{Kock} and \cite{RobertsUnpub}.

%%%%%%%%%%%%%%%%%%%%%%%
%%%%%%%%%%%%%%%%%%%%%%%
\section{Adams operations on matrix factorizations}
\label{operations}
%%%%%%%%%%%%%%%%%%%%%%%
%%%%%%%%%%%%%%%%%%%%%%%

In this section, we define cyclic Adams operations on matrix factorizations, closely following the construction of cyclic Adams operations on perfect complexes
with support found in Sections 2 and 3 of \cite{brown2016cyclic}. We prove these operations enjoy analogues of many of the key properties of the operations on perfect complexes with support constructed in loc. cit.

%%%%%%%%%%%%%%%%%%%%%%%
%%%%%%%%%%%%%%%%%%%%%%%
\subsection{Construction}
\label{construction}
%%%%%%%%%%%%%%%%%%%%%%%
%%%%%%%%%%%%%%%%%%%%%%%

Let $Q$ be a Noetherian commutative ring, $f \in Q$ any element (including possibly $f = 0$), and $G$ a finite group.
Let $\mf(Q,f;G)$ be the category of {\em $G$-equivariant matrix factorizations}. When $G$ is the trivial group, 
this is the category described in the introduction. More generally, 
an object of $\mf(Q,f;G)$ is an object $P$ of $\mf(Q,f)$ equipped with a $G$-action (i.e., a group
homomorphism $G \to \Aut_{\mf(Q,f)}(P)$), and a morphism is a $G$-equivariant morphism of matrix factorizations. 

The category $\mf(Q,f;G)$ is an exact category, with the notion of exactness given degree-wise in the evident manner. 

\begin{rem} We could equivalently  define an object of $\mf(Q,f;G)$ to consist of a pair of $Q[G]$-modules $P_0, P_1$ that are finitely generated and projective
  as $Q$-modules, together with 
 a pair of morphisms of $Q[G]$-modules, $(\a: P_1 \to P_0, \b: P_0 \to P_1)$, such that $\a \b$ and $\b
  \a$ are each multiplication by $f$ (which is central in $Q[G]$). Moreover, if $|G|$
  is invertible in $Q$, we have $\mf(Q,f;G) = \mf(Q[G],f)$. 
\end{rem}

\begin{ex} If $f = 0$ (and $G$ is trivial), $\mf(Q,0)$ is the category of $\Z/2$-graded complexes of finitely generated projective $Q$-modules, with morphisms
  being chain maps. 
\end{ex}

A {\em homotopy} joining morphisms $g_1, g_2: P \to P'$ in $\mf(Q,f;G)$ is defined just as in the introduction, with the added condition that it be
$G$-equivariant. In detail, it is 
a $Q$-linear, $G$-equivariant map $h: P \to P'$ of degree $1$ such that $d_{P'} h + h d_P = g_1
- g_2$. The {\em homotopy category} of $\mf(Q,f;G)$ is the category $[\mf(Q,f;G)]$ obtained from $\mf(Q,f;G)$ by identifying homotopic morphisms.

Given a ring homomorphism $Q \to Q'$ sending $f$ to $f'$, 
there is an evident functor $\mf(Q,f;G) \to \mf(Q', f';G)$ given by extension of scalars along $Q \to Q'$. 
When $Q' = Q_\fp$ for $\fp \in \Spec(Q)$, we write this functor as $P \mapsto P_\fp$.

For an object $P \in \mf(Q,f;G)$, define the {\em support} of $P$ to be
$$
\supp(P) = \{ \fp \in \Spec(Q) \, | \, \text{$P_\fp$ is not homotopy equivalent to $0$ in $\mf(Q_\fp, f;G)$} \}.
$$
Given a closed subset $Z$ of $\Spec(Q)$, define $\mf^Z(Q,f;G)$ to be the full subcategory of $\mf(Q,f)$ consisting of objects $P$ satisfying $\supp(P) \subseteq
Z$. Note that $\mf^Z(Q,f;G)$ is a full, exact subcategory of $\mf(Q,f;G)$, and $[\mf^Z(Q,f;G)]$ is a full  subcategory of $[\mf(Q,f;G)]$.

We will mainly use the notion of supports for matrix factorizations when $f = 0$ and $G$ is trivial, in which case objects of $\mf(Q,0)$ are ($\Z/2$-graded) complexes. 
One must be careful in this situation not to conflate the notion of being homotopy equivalent to $0$ with being acyclic. 
The former implies the latter, but the latter does not imply the former in general. These conditions are equivalent, however, in the following case:

\begin{lem} \label{lem728d}
If $Q$ is a regular ring, an object $P \in \mf(Q,0)$ is contractible if and only if $H_0(P) = H_1(P) = 0$. 
\end{lem}

\begin{proof}
Suppose $P=(\alpha_0: P_0 \to P_1, \alpha_1: P_1 \to P_0)$ is acyclic, and
set $M = \ker(\alpha_1) = \im(\alpha_0)$ 
and $N = \ker(\alpha_0) = \im(\alpha_1)$. 
We claim that $M$ and $N$ are projective.
It suffices to prove $M_\mathfrak{p}$ and $N_\mathfrak{p}$ are free for all primes $\mathfrak{p}$. Since
$$
0 \lra M_\fp \lra (P_1)_{\fp} \lra (P_0)_\fp \lra (P_1)_\fp \lra \cdots 
$$
is exact, we see that, for any $d$, $M_\fp$ is a $d$-th syzygy 
of some other $Q_\fp$-module. Taking $d > \dim(Q_\fp)$ gives that $M_\fp$ is free. Similarly, $N$ is projective.

Choose splittings $\pi_0: P_0 \to
N$ and $\pi_1: P_1 \to M$ of the inclusions $N \into P_0$ and $M \into
P_1$. Define $A: P_0 \to N \oplus M$ 
and $B: P_1 \to N \oplus M$ to be given by
$\begin{pmatrix} \pi_0  \\ \alpha_0
  \\ \end{pmatrix}$ and $\begin{pmatrix} \alpha_1 \\   \pi_1
  \\ \end{pmatrix} $, respectively. Set $E:= \begin{pmatrix} 0 & 0 \\ 0 & 1 \end{pmatrix}$ and $F:=\begin{pmatrix} 1 & 0 \\
  0 & 0 \end{pmatrix}$. 

We have the following isomorphism of matrix factorizations
$$
\xymatrix{
P_0 \ar[r]^{\alpha_0} \ar[d]_{A} & P_1 \ar[d]_{B} \ar[r]^{\alpha_1} &
P_0 \ar[d]_{A} \\
N \oplus M \ar[r]^{E} &  N
\oplus M \ar[r]^{F} & N \oplus M \\
}
$$
and the bottom matrix factorization is clearly contractible.

\end{proof}

\begin{rem} 
When $Q$ is regular, $f$ is a non-zero-divisor, and $G$ is trivial, the support of any object of $\mf(Q,f)$ is a subset of 
$$
\Sing(R) := \{ \fp \in \Spec(R) \, | \, \text{$R_\fp$ is not regular}\} 
$$
where $R = Q/(f)$, and where we identify $\Spec R$ with its image in $\Spec Q$. 
Thus, in this case, we have
$$
\mf(Q,f) = \mf^{\Sing(R)}(Q,f).
$$
Eventually, we will be making the additional assumption that $R$ is an isolated singularity, meaning $Q$, and hence $R$, is local, and 
$\Sing(R) = \{\fm\}$.
\end{rem}

Define the Grothendieck group $K_0(\mf^Z(Q,f;G))$ to be the abelian monoid given by isomorphism classes of objects of
$\mf^Z(Q,f;G)$ under the operation of direct sum, modulo the relations $[P] = [P'] + [P'']$ if there exists a short exact sequence $0 \to P' \to P \to P'' \to
0$  and
$[P] = [P']$ if $P$ and $P'$ are homotopy equivalent. As with the $K$-theory of complexes, 
$K_0(\mf^Z(Q,f;G))$ is an abelian group, since $[P] + [\Sigma(P)] = 0$,
where $\Sigma(P)$ denotes the suspension of $P$. 

For $P \in \mf(Q,f;G)$ and $P' \in \mf(Q, f';G')$, the tensor product $P \otimes_Q P'$ is the usual tensor product of $Q$-modules,
with grading determined  by $|p \otimes p'| = |p| + |p'|$ and differential $\del(p \otimes p') = d_P(p) \otimes p' + (-1)^{|p|} p \otimes d_{P'}(p')$. The group
$G \times G'$ acts in the evident manner, and 
the resulting object belongs to $\mf(Q, f + f'; G \times G')$, since $\del^2$ is multiplication by $f + f'$. 
Note, in particular, that the $n$-th tensor power of an object of $\mf(Q,f)$ belongs to $\mf(Q, nf)$.

We proceed to define cyclic Adams operations on $K_0(\mf^Z(Q,f))$. The construction is closely parallel to that for $K_0^Z(Q)$ given in \cite{brown2016cyclic}, 
with one minor exception: the need to ``divide by $p$''.

For an integer $n \geq 1$, we define a functor
$$
T^n: \mf^Z(Q, f) \to \mf^Z(Q, nf; \Sigma_n)
$$
given, on objects, by sending 
$P \in \mf^Z(Q,f)$ to the matrix factorization 
$$
T^n(P) = \overbrace{P \otimes_Q \cdots \otimes_Q P}^{\text{$n$ times}}
$$ 
equipped with the left action of $\Sigma_n$ given by
$$
\sigma(p_1 \otimes \cdots \otimes p_n) = \pm p_{\sigma^{-1}(1)} \otimes \cdots \otimes p_{\sigma^{-1}(n)}.
$$
The sign is uniquely determined by the following rule: if $\sigma$ is the transposition $(i \text{ } i+1)$ for some $1 \le i \le n-1$ and $p_1, \dots, p_n$ are homogenous elements of $P$, then
$$\sigma(p_1 \otimes \cdots \otimes p_n) = (-1)^{|p_i| |p_{i+1}|} p_1 \otimes \cdots p_{i-1} \otimes p_{i+1} \otimes p_i \otimes p_{i+2} \otimes \cdots \otimes p_n.$$

The rule for morphisms is the evident one. 

Following Section 2 of \cite{brown2016cyclic}, for any $i,j$, let $\Sigma_{i,j}$ be the image of the canonical homomorphism $\Sigma_i \times \Sigma_j \into
\Sigma_{i+j}$, and define a pairing
$$
\star_{i,j}: K_0(\mf^Z(Q, if); \Sigma_i) \times
K_0(\mf^Z(Q, jf); \Sigma_j) \to K_0(\mf^Z(Q, (i+j)f); \Sigma_{i+j})
$$
induced by the bi-functor $(P, P') \mapsto Q[\Sigma_{i+j}] \otimes_{Q[\Sigma_{i,j}]} P \otimes_Q P'$. 
This pairing is well-defined, commutative, and associative, by an argument identical to the proof of Lemma 2.4 in
\cite{brown2016cyclic}. 

The proof of Theorem 2.2 in \cite{brown2016cyclic} also holds nearly verbatim for matrix
factorizations and leads to a proof of: 

\begin{thm} \label{thm1030}
For a commutative Noetherian ring $Q$, closed subset $Z$ of $\Spec(Q)$, element $f \in Q$, and integer $n \geq 1$, there is a function
$$
t^n_\Sigma: K_0(\mf^Z(Q,f)) \to K_0(\mf^Z(Q, nf; \Sigma_n)) 
$$
such that, for an object $P \in \mf^Z(Q,f)$, we have
$$
t^n_\Sigma([P]) = [T^n(P)].
$$
\end{thm}

\begin{rem} 
As in \cite[\S 5]{brown2016cyclic}, if $k$ is a positive integer such that  $k!$ is invertible in $Q$, then 
one can use Theorem \ref{thm1030} to establish an operation $\lambda^k$ on $K_0(\mf^Z(Q,f))$ that is induced from the $k$-th exterior power functor. Since 
we won't  use such operations in this paper, we omit the details.
\end{rem}

We now assume $p$ is a prime that is invertible in $Q$, and we define $C_p$ to be the subgroup of $\Sigma_p$ generated by the $p$-cycle $(1 \, 2 \, \cdots\, p)$. 
For any $p$-th root of unity $\zeta$ belonging to $Q$ 
(including the case $\zeta = 1$), let $Q_{\zeta}$ 
denote the $Q[C_p]$-module $Q$ equipped with the $C_p$-action $\s q = \zeta q$. 
For $P \in \mf^Z(Q, pf; C_p)$, we define
$$
P^{(\zeta)} := \Hom_{Q[C_p]}(Q_\zeta, P) = 
\ker(\s - \zeta: P \to P).
$$
Since $p$ is invertible and $\zeta$ belongs to $Q$, 
the module $Q_{\zeta}$ is a direct summand of $Q[C_p]$, and so $P \mapsto P^{(\zeta)}$ is an exact functor. It therefore induces a map
$$
\phi^p_\zeta: K_0(\mf^Z(Q, pf; C_p)) 
\xra{[P] \mapsto [P^{(\zeta)}]} 
K_0(\mf^Z(Q, pf)),
$$
and so we may form
the composition 
$$
K_0(\mf^Z(Q,f)) \xra{t^p_{\Sigma}} K_0(\mf^Z(Q, pf; \Sigma_p)) 
\xra{\res}
K_0(\mf^Z(Q, pf; C_p)) 
\xra{\phi^p_\zeta}
K_0(\mf^Z(Q, pf)).
$$ 
We come upon the need to ``divide by $p$''. In general,
if $u \in Q$ is a unit, we
define an auto-equivalence 
$$
\mult_u: \mf^Z(Q, f) \to \mf^Z(Q, uf)
$$
by sending a matrix factorization $(\a, \b)$ to $(\a, u \b)$.
(Its inverse is given by $\mult_{u^{-1}}$.) 
For example, in Section \ref{sec-dao} below, we will employ the functor $\mult_{-1}$, which we will write as $\mult_{-1}(P) = P^\circ$. Here, we use
$\mult_{\frac{1}{p}}$, and we define $t^p_\zeta$ to be the composition
$$
K_0(\mf^Z(Q,f)) \xra{\phi^p_\zeta \circ \res \circ t_\Sigma^p}
K_0(\mf^Z(Q, pf)) \xra{\mult_{\frac{1}{p}}}
K_0(\mf^Z(Q, f)).
$$ 

Let $A_p$ denote the subring of $\C$ given by $\Z[\frac{1}{p}, e^{\frac{2\pi i}{p}}]$.

\begin{defn} \label{def930} Assume $p$ is a prime, $Q$ is a (commutative, Noetherian) $A_p$-algebra, $f$ is any element of $Q$, and $Z$ is a closed subset of
  $\Spec(Q)$. Define
$$
\cPsi^p = \sum_{\zeta} \zeta t^p_\zeta:  K_0(\mf^Z(Q,f)) \to K_0(\mf^Z(Q,f)),
$$
where the sum ranges over all $p$-th roots of unity. (In this formula, the $\zeta$ occurring as a coefficient is interpreted as belonging to $\Z[e^{2\pi i/p}]$
whereas 
the $\zeta$ occurring as a subscript denotes its image in $Q$ under the map $A_p \to Q$.) 
\end{defn}

\begin{rem}
The image of $\cPsi^p$ is, \emph{a priori}, contained in the group 
$K_0(\mf^Z(Q,f)) \otimes_\Z \Z[e^{\frac{2\pi i}{p}}]$. 
But, by an argument identical to the proof of
Corollary 3.5 in \cite{brown2016cyclic}, we have
$$
\sum_\zeta \zeta t^p_\zeta = t_1^p - t_{\zeta'}^p
$$
for any fixed primitive $p$-th root of unity $\zeta'$, and thus the image of $\cPsi^p$ can be taken to be $K_0(\mf^Z(Q,f))$.
\end{rem}

\begin{rem} \label{phiremark}
Setting $\phi^p = \sum_{\zeta} \zeta \phi^p_\zeta$, one gets another formulation: 
$$
\cPsi^p = \mult_{\frac{1}{p}} \circ \phi^p \circ \res \circ t^p_\Sigma.
$$
\end{rem}

%%%%%%%%%%%%%%%%%%%%%%%%%%
%%%%%%%%%%%%%%%%%%%%%%%%%%
\subsection{Axioms for Adams operations on matrix factorizations \`a la Gillet-Soul\'e}
%%%%%%%%%%%%%%%%%%%%%%%%%%
%%%%%%%%%%%%%%%%%%%%%%%%%%
In this subsection, we show the operations $\cPsi^p$ satisfy the following analogues of the axioms of Gillet and Soul\'e (cf. Theorem 3.7 in \cite{brown2016cyclic}):

\begin{thm} \label{psi-GS-axioms-mf}
Assume $p$ is a prime, $Q$ is a (commutative, Noetherian) $A_p$-algebra, 
$f, f_1, f_2$ are any elements of $Q$, and $Z$ is a closed subset of $\Spec(Q)$.
\begin{enumerate}

\item $\cPsi^p$ is a group endomorphism of $K_0(\mf^Z(Q,f))$.

\item For $\a \in K_0(\mf^Z(Q,f_1))$ and  $\b \in K_0(\mf^W(Q,f_2))$,
$$
\cPsi^p(\a \cup \b) = \cPsi^p(\a) \cup \cPsi^p(\b) \in K_0(\mf^{Z \cap W}(Q,f_1 + f_2)),
$$
where $\cup$ is the multiplication rule on Grothendieck groups induced by tensor product. The three operators $\cPsi^p$ in the equation are, from left
to right, acting on $K_0(\mf^{Z\cap W}(Q,f_1+f_2))$, $K_0(\mf^Z(Q,f_1))$, and $K_0(\mf^W(Q,f_2))$.

\item $\cPsi^p$ is functorial in the following sense: Suppose $\rho: Q \to Q'$ is map of $A_p$-algebras, $f' = \rho(f)$, and $\widetilde{\rho}^{-1}(Z) \subseteq Z'$ 
where $\widetilde{\rho} \colon \Spec Q' \to \Spec Q$ is the induced map on spectra.  
Then extension of scalars along $\rho$ induces a map
$K_0(\mf^Z(Q,f)) \to K_0(\mf^{Z'}(Q',f'))$ that commutes with the actions of $\cPsi^p$.

\item 
If $f = gh$, so that $(g,h) := (Q \xra{g} Q,  Q \xra{h} Q)$ is an object of $\mf^{V(g,h)}(Q, f)$, we have
$$
\cPsi^p[(g,h)] = p [(g,h)].
$$

\end{enumerate}
\end{thm}

\begin{proof}
The proofs of (1)--(3) are essentially identical to the proofs of parts (1)--(3) of Theorem 3.7 in \cite{brown2016cyclic}.
As for (4), let $(0,0)$ denote the matrix factorization $(Q \xra{0} Q, Q \xra{0} Q)$ of 0, and let $X$ denote the tensor product 
$$
(g, ph) \otimes_Q (0,0) \otimes_Q \cdots \otimes_Q (0,0).
$$

Set $\zeta:=e^{\frac{2 \pi i}{p}}$ and $\sigma:=(1 \, 2 \, \cdots \, p) \in C_p$. 
We equip $X$ with a $C_p$ action by letting $\sigma$ act on the $i$-th factor of $X$ in the following way: if $x$ has odd degree, $\sigma \cdot x = \zeta^{i-1}x$; if $x$ has even degree, $\sigma \cdot x = x$.

We claim that there is an isomorphism
$$
T^p([g,h]) \cong (g, ph) \otimes_Q (0,0) \otimes_Q \cdots \otimes_Q (0,0) 
$$
in $\mf^{V(g, h)}(Q, pf; C_p)$. To prove the claim, let $V$ be a free $Q$-module of rank $p$ 
with a fixed basis $\{e_0, \dots, e_{p-1}\}$.
We identify the underlying $Q$-modules of $T^p((g,h))$
and $X$ with the exterior algebra 
$\bigwedge V$ of $V$; under this identification, the action of $C_p$ on $T^p((g,h))$ is given by 
$$
\sigma (e_{i_1} \wedge \cdots \wedge e_{i_n} )=  e_{\sigma^{-1}(i_1)} \wedge \cdots \wedge e_{\sigma^{-1}(i_n)}, 
$$
and the action of $C_p$ on $X$ is given by
$$
\sigma (e_{i_1} \wedge \cdots \wedge e_{i_n} )= \zeta^{i_1 +
\cdots + i_n} e_{i_1} \wedge \cdots \wedge e_{i_n}.
$$

For $0 \le i \le p-1$, define $v_i:=\frac{1}{p} \sum_j \zeta^{ij}
e_j$. Then $v_0, \dots,  v_{p-1}$ form a basis of $V$. Let $\alpha: \bigwedge
V \to \bigwedge V$ denote the
$Q$-algebra automorphism given by $e_i \mapsto v_i$. 
Then $\alpha$ yields an
isomorphism $T^p((g,h)) \xra{\cong} X$ of $C_p$-equivariant matrix
factorizations; this proves the claim. 

(In checking the details here, it is useful to note the following: the
``differential'' on $T^p((g,h))$ is given by $s_0 + s_1$, where $s_0$
is left-multiplication by $h(e_0 + \cdots +
e_{p-1})$, and $s_1$ is given by the Koszul differential on the sequence
$(g, g, \dots, g)$. Similarly, the ``differential'' on $X$ is given
by $t_0 + t_1$, where $t_0$ is left-multiplication by $phe_0$ and
$t_1$ is given by the Koszul differential on the sequence $(g, 0,
\dots, 0).$)

By Remark~\ref{phiremark}, and the result analogous to 
Lemma 3.11 of \cite{brown2016cyclic} for matrix factorizations (with essentially the same proof), we have 
$$
\cPsi^p ([(g,h)]) 
= \mult_{\frac{1}{p}} \! \left(  \phi^p ([(g,ph)]) \cup \phi^p ([(0,0)]) \cup \cdots \cup \phi^p ([(0,0)])\right). \\
%\mult_{\frac{1}{p}} [(g,ph)] \prod_{i=1}^{p-1} (1 - \zeta^{-i}) =  p [(g,h)].
$$
$\phi^p$ acts as the identity on the first factor, which is equipped with the trivial action of $C_p$. 
Furthermore, direct calculation on the $(i+1)$-st factor yields  
$$
\phi^p([(0,0]) =  [I] + \zeta^{i} [\Sigma I] = (1-\zeta^{i}) [I]
$$
where $I$ denotes the unit matrix factorization $(0 \xra{0} Q,Q \xra{0} 0)$. 
Thus, one obtains 
$$
\cPsi^p ([(g,h)]) 
=  \mult_{\frac{1}{p}} \! \left( [(g,ph)] \cup [I] \cup \cdots \cup [I] \right) 
\prod_{i=1}^{p-1} (1 - \zeta^{i})
= p [(g,h)],
$$
since $\prod_{i=1}^{p-1} (1 - \zeta^{i}) = p$.
\end{proof}

\begin{cor} \label{cor728}
If $a = (a_1, \dots, a_n)$ is a 
sequence of elements in an $A_p$-algebra $Q$, 
and $K(a)$ is the associated $\Z/2$-folded Koszul complex, 
regarded as an object of $\mf^{V(a_1, \dots, a_n)}(Q, 0)$, then
$$
\cPsi^p([K(a)]) = p^n [K(a)] \in K_0(\mf^{V(a_1, \dots, a_n)}(Q, 0)).
$$
\end{cor}

\begin{proof}
This follows from parts (2) and (4) of the Theorem, because 
$K(a)$ is the tensor product of the matrix factorizations $(a_i,0)$ and $\Z/2$-folding commutes with tensor product. 

\end{proof}

%%%%%%%%%%%%%%%%%%%%%%%%%%
%%%%%%%%%%%%%%%%%%%%%%%%%%
\subsection{Diagonalizability}
%%%%%%%%%%%%%%%%%%%%%%%%%%
%%%%%%%%%%%%%%%%%%%%%%%%%%

Suppose $Q$ is a regular ring and $f \in Q$ is a non-zero-divisor. Recall, from the introduction, that $\cP^{V(f)}(Q)$ denotes the category of bounded complexes of finitely generated and projective $Q$-modules whose homology is supported on $V(f)$, and $K_0^{V(f)}(Q)$ denotes its Grothendieck group. In this subsection, we construct a surjection
$$\rho_f: K_0^{V(f)}(Q) \onto K_0(\mf(Q,f))$$
that commutes with the actions of $\cPsi^p$. Using this, and Corollary 3.12 of \cite{brown2016cyclic} 
(the proof of which is really due to Gillet-Soul\'e), we deduce that the action of 
$\cPsi^p$ on $K_0(\mf(Q,f))_\Q$ decomposes the latter into eigenspaces of the expected weights.

Let $K_f$ denote the Koszul dga associated to $f$, so that, as a $Q$-algebra, $K_f = Q[\e]/(\e^2)$ with $|\e| = 1$, and it is equipped with the $Q$-linear differential $d$
satisfying $d(\e) = f$. Let $P(K_f/Q)$ denote the full subcategory of the category of dg-$K_f$-modules consisting of those that 
are finitely generated and projective as $Q$-modules. An object of
$P(K_f/Q)$ is thus a bounded complex $P$ of finitely generated projective $Q$-modules equipped with a degree one $Q$-linear map $s: P_\cdot \to P_{\cdot + 1}$
satisfying 
$d_P s + s d_P = f$ and $s^2 = 0$. (The map $s$ is given by multiplication by $\e$.) A morphism from $(P, d_P, s)$ to $(P',d_{P'},s')$ is a chain map $g$ such that
$g s = s' g$. A homotopy from $g_1$ to $g_2$ is a degree one map $h$ such that $d_{P'}h + hd_P = g_1 - g_2$ and $hs = s'h$.

There are functors
$$
\cP^{V(f)}(Q) \xla{F} P(K_f/Q) \xra{\Fold} \mf(Q,f),
$$
where $F$ is the forgetful functor that sends $(P,d_P, s)$ to $(P, d_P)$, and $\Fold$ sends $(P, d, s)$ to the following matrix factorization:
the even degree part is $\bigoplus_i P_{2i}$, the odd degree part is $\bigoplus_{i} P_{2i+1}$ and the degree one
endomorphism is $\del := d + s$.

Define $K_0(P(K_f/Q))$ to be the Grothendieck group of objects modulo relations coming from short exact sequences and homotopy equivalences as usual.

\begin{lem} If $f$ is a non-zero-divisor in a regular ring $Q$, 
the functor $F$ induces an isomorphism
$$
K_0(P(K_f/Q)) \xra{\cong} K_0^{V(f)}(Q).
$$
\end{lem}

\begin{proof} 
Let $R = Q/(f)$. One has an evident quasi-isomorphism $K_f \xra{\sim} R$ of dga's, and hence an equivalence of triangulated categories $D^b(R) \xrightarrow{\cong} D^b(K_f)$
induced by restriction of scalars. Thus, one has an isomorphism
$$
G_0(R) = K_0(D^b(R)) \xrightarrow{\cong} K_0(D^b(K_f)).
$$
We may model $D^b(K_f)$ by semi-projective $K_f$-modules with finitely generated homology. 
Since
$Q$ is regular, the good truncation of such a complex in sufficiently high degree is a complex of projective $Q$-modules. It thus follows from Quillen's
resolution theorem that the
inclusion map determines an isomorphism 
$$
K_0(P(K_f/Q)) \xra{\cong} K_0(D^b(K_f)).
$$
We thus obtain an isomorphism
$G_0(R) \xra{\cong} K_0(P(K_f/Q))$, which we can describe explicitly as follows:  
if $M$ is a finitely generated $R$-module, form a (possibly infinite) $K_f$-semi-projective resolution $P \xra{\sim} M$ of $M$. Then the map sends $[M]$ to
$[P']$ where $P'$ is a good truncation
of $P$ in sufficiently high degree.

We also have the more classical isomorphism $G_0(R) \xra{\cong} K_0^{V(f)}(Q)$, sending $[M]$ to the class of a $Q$-projective resolution of $M$. Since the complex $P'$
constructed above is an example of such a resolution, it is clear that 
the triangle 
$$
\xymatrix{
K_0(P(K_f/Q)) \ar[rr]^-{F} && K_0^{V(f)}(Q) \\
& G_0(R) \ar[ul]^\cong \ar[ur]^\cong 
}
$$
commutes.
\end{proof}

The functor $\Fold$ induces a map from 
$K_0(P(K_f/Q))$ to $K_0(\mf(Q,f))$, 
and thus, using the lemma, we obtain the desired map
$\rho_f: K_0^{V(f)}(Q) \to K_0(\mf(Q,f))$. 
Explicitly, the construction shows that if an object 
$P \in \cP^{V(f)}(Q)$ admits a degree one map $s$ satisfying $ds + sd = f$ and
$s^2 = 0$, then $\rho_f([P]) = [\Fold(P,d,s)]$.  
In particular, the map $\rho_f$ is surjective, since for a matrix factorization 
$(\alpha: P_1 \to P_0, \beta: P_0 \to P_1) \in \mf(Q,f)$, we have 
$(\alpha, \beta) = \Fold(P, \alpha, \beta)$. 

Since there exists an isomorphism $G_0(Q/(f)) \xrightarrow{\cong} K_0^{V(f)}(Q)$ which sends the class of a finitely generated $Q/(f)$-module to the class of a chosen $Q$-projective resolution of it,
we obtain a surjective map
$$
G_0(Q/(f)) \onto K_0(\mf(Q,f)).
$$
Note that this surjection agrees with the one induced by the inverse of the equivalence $[\mf(Q,f)] \xrightarrow{\cong} D_{\sing}(Q/(f))$ from
Theorem~\ref{introthm1} of the introduction.

Given a finitely generated $Q/(f)$-module $M$, let $[M]_\stable \in K_0(\mf(Q,f))$ denote the image of $[M]$ under the above surjection
$G_0(Q/(f)) \onto K_0(\mf(Q,f))$.
Explicitly, for such an $M$, one may find a $Q$-projective resolution $(P, d)$ of it for which
there exists a degree one endomorphism $s$ of $P$ satisfying $ds + sd = f$ and $s^2 = 0$ (by taking, for instance, as above, a good truncation in sufficiently high degree of 
a $K_f$-semi-projective resolution $P \xra{\sim} M$). 
Then $[M]_\stable = [\Fold(P, d, s)]$.

We will use the following result to deduce the diagonalizability 
of $\cPsi^p$ on the Grothendieck group of matrix factorizations 
from the corresponding result for complexes. 

\begin{prop} \label{prop727}
Assume $Q$ is a regular $A_p$-algebra and $f \in Q$ is a non-zero-divisor.
The map $\rho_f$ commutes with the Adams operations $\cPsi^p$.
\end{prop}

\begin{proof} 
We need to show the diagram
$$
\xymatrix{
K_0^{V(f)}(Q) \ar[r]^-{\rho_f} \ar[d]^{\cPsi^p} & K_0(\mf(Q,f)) \ar[d]^{[Y] \mapsto [T^p(Y)^{(1)}] - [T^p(Y)^{(\zeta)}]}  \\
K_0^{V(f)}(Q) \ar[r]^-{\rho_{pf}} \ar[d]^= & K_0(\mf(Q,pf)) \ar[d]^{\mult_{\frac{1}{p}}}  \\
K_0^{V(f)}(Q) \ar[r]^-{\rho_{f}} & K_0(\mf(Q,f))   \\
}
$$
commutes.

It suffices to check the commutativity of the top square on classes $[P]$ for which there exists an $s$ with $ds + sd = f$ and $s^2 = 0$.
Recall that the induced differential $T^p(d)$ on $T^p(P)$ is given by 
$$
T^p(d)(x_1 \otimes \cdots \otimes x_p) = 
\sum_{i=1}^p (-1)^{|x_1| + \cdots + |x_{i-1}|} x_1 \otimes \cdots \otimes d(x_i) \otimes \cdots \otimes x_p,
$$
and we define $T^p(s)$ to be the degree one map given by the same formula with $s$ in place of $d$. 
Then $T^p(d) T^p(s) + T^p(s) T^p(d) = pf$ and $T^p(s)^2 = 0$.  Moreover,
it follows from the definitions that there is a canonical isomorphism
$$
T^p(\Fold(P,d,s)) \cong \Fold(T^p(P),T^p(d), T^p(s)) \in \mf(Q, pf),
$$
and this isomorphism is equivariant for the action of $\Sigma_p$. The commutativity of the top square in the diagram follows.

The bottom square commutes by the more general lemma below.
\end{proof}

\begin{lem} \label{lem730}
If $Q$ is a regular, 
$f \in Q$ is a non-zero-divisor,  and $u \in Q$ is a unit, 
the triangle
$$
\xymatrix{
&& K_0(\mf(Q,f)) \ar[dd]^{\mult_u} \\
K_0^{V(f)}(Q) \ar@{->>}[rru]^{\rho_f} \ar@{->>}[rrd]^{\rho_{uf}} & \\
  && K_0(\mf(Q,uf)) 
}
$$
commutes.
\end{lem}

\begin{proof} 
Again, it suffices to check the commutativity of the diagram on
classes $[P]$ such that $P$ is a complex with differential $d$ for which there exists an $s$ with $ds + sd = f$ and $s^2 = 0$. 
If $[P]$ is such a class, $\rho_f([P])= [\Fold(P, d,s)]$. 

Before applying $\rho_{uf}$, first replace $(P,d)$ 
by the isomorphic complex $(P',d')$ 
with $P'_i = P_i$ for all $i$ 
and with $d'_i = d_i$ for $i$ odd and $d'_i = u d_i$ for $i$ even. 
Defining $s'$ as $s'_i = s_i$ for $i$ odd and 
$s'_i = us_i$ for $i$ even, one has $d's'+s'd' = uf$. 
Then $\rho_{uf}([P]) = [\Fold(P,' d', s')] 
= \mult_u ([\Fold(P, d,s)]) = (\mult_u \circ \rho_f)([P])$.
\end{proof}

\begin{thm}\label{psi-eigenspaces} Assume $Q$ is a regular $A_p$-algebra of dimension $d$ and $f \in Q$ is a non-zero-divisor.
There is a decomposition
$$
K_0(\mf(Q,f))_\Q
= 
\bigoplus_{i =1}^d  K_0(\mf(Q,f))_\Q^{(i)},
$$
which is independent of $p$, 
such that $\cPsi^p$ acts on $K_0(\mf(Q,f))_\Q^{(i)}$ as multiplication by $p^i$. 
Moreover, for a finitely generated $Q/(f)$-module $M$, we have
$$
[M]_\stable \ 
\in 
\bigoplus_{i=\codim_{Q/(f)} M +1}^d  K_0(\mf(Q,f))_\Q^{(i)}.
$$
\end{thm}

\begin{proof} 
This follows from Corollary 3.12 of \cite{brown2016cyclic} and Proposition \ref{prop727} by defining 
$K_0(\mf(Q,f))_\Q^{(i)}$ to be the image of  $K_0^{V(f)}(Q)_\Q^{(i)}$ under $\rho_f \otimes \Q$. 
\end{proof}

We close this subsection with a technical result needed below.

\begin{cor} \label{cor727c}
If $Q$ is a regular $A_p$-algebra for a prime $p$, $f \in Q$ is a non-zero-divisor,  and $u \in Q$ is a unit, 
we have an equality of maps $\cPsi^p \circ \mult_u = \mult_u \circ \cPsi^p$ from
  $K_0(\mf(Q,f))$ to $K_0(\mf(Q,uf))$. 
\end{cor} 

\begin{proof} By Proposition \ref{prop727}, the diagonal maps in the commutative diagram of Lemma \ref{lem730} commute with the action of $\cPsi^p$, and these maps are surjective. 
\end{proof}

%%%%%%%%%%%%%%%%%%%%%%%
\section{Dao-Kurano's Conjecture}
\label{DK}
%%%%%%%%%%%%%%%%%%%%%%%

In this section, we apply the results of Section~\ref{operations} to
give a proof of Theorem~\ref{introthm5} from the introduction. 

%%%%%%%%%%%%%%%%%%%%%%%%%%%%%%%
\subsection{Some properties of $\Z/2$-graded complexes}
\label{properties}
%%%%%%%%%%%%%%%%%%%%%%%%%%%%%%%

We will need some general results about $\Z/2$-graded complexes. Much of what we need holds in great generality, and so we start by working over a
Noetherian commutative ring $B$. 

Let $\LF(B,0)$ denote the abelian category of all $\Z/2$-graded complexes of $B$-modules (``LF'' stands for ``linear
factorization''), and let $\lf(B,0)$ denote the full subcategory of $\LF(B,0)$ consisting of complexes whose components are finitely generated $B$-modules. An object of $\LF(B,0)$ consists of a pair of $B$-modules, $M^0$ and $M^1$, together with maps $d^0: M^0 \to M^1$ and $d^1: M^1 \to M^0$ such that $d^1 \circ d^0 = 0 = d^0 \circ d^1$. 
Morphisms are given by the evident
$\Z/2$-graded analogues of chain maps. We also have the evident $\Z/2$-versions of quasi-isomorphisms and homotopies of chain maps.  
For objects $X, Y \in \LF(B,0)$, let $\HomLF(X,Y)$ denote the $\Z/2$-analogue of the mapping complex construction. So $\HomLF(X,Y) \in \LF(B,0)$ with
$\HomLF(X,Y)^\e = \bigoplus_{\e' + \e'' = \e} \Hom_B(X^{\e'}, Y^{\e''})$. Note that the zero cycles in $\HomLF(X,Y)$ are, by definition, the set of morphisms from
$X$ to $Y$ in $\LF(B,0)$, and $H^0 \HomLF(X,Y)$ is the set of morphisms
modulo homotopy.

We write $X \otimes_{\LF} Y \in \LF(B,0)$ 
for the evident $\Z/2$-graded analogue of
the tensor product of complexes, so that
$$
(X \otimes_{\LF} Y)^\e = \bigoplus_{\e = \e' + \e''} X^{\e'}
  \otimes_B Y^{\e''}.
$$
We will also need the notion of the totalization $\Tot(X_\cdot)$ of a bounded complex
$$
X_\cdot := (0 \to X_m \to \cdots \to X_0 \to 0)
$$
of objects of $\LF(B,0)$, defined in a manner similar to the $\Z$-graded
setting.
In more detail, we have 
$$
\Tot(X_\cdot)^\e = \bigoplus_{i = 0}^m X_i^{i +\e},
$$
with superscripts taken modulo $2$.  Moreover, if 
$$
0 \to X_m \to \cdots \to X_0 \to M \to 0
$$
is an exact sequence in $\LF(B,0)$, then there is 
a natural quasi-isomorphism 
$$
\Tot(X_\cdot) \xra{\sim} M 
$$
in $\LF(B,0)$.

For $M \in \LF(B,0)$, define 
$Z(M)$ to be the $\Z/2$-graded module consisting of the kernels 
of the two maps comprising the complex $M$, and define 
$B(M)$ to be the $\Z/2$-graded module given by the images 
of the two maps comprising $M$. Let $\Hlf(M)$ denote the 
$\Z/2$-graded module consisting of the homology modules of $M$. Each of $B$, $Z$, and $H$ can be 
interpreted as a functor from $\LF(B,0)$ to
itself, and they restrict to functors from $\lf(B,0)$ to itself. Note that $B(M) \subseteq Z(M)$ and $\Hlf(M) = Z(M)/B(M)$. 

Recall that $\mf(B,0)$ is the full subcategory of $\lf(B,0)$ consisting of complexes whose components are projective $B$-modules. 

\begin{defn} \label{defCE} 
An object $X \in \mf(B,0)$ is called {\em proper} if $Z(X)$, $B(X)$ 
and $\Hlf(X)$ are all projective $R$-modules.

For $M \in \lf(B,0)$, an exact sequence of the form 
$$
\dots \to X_m \to \cdots \to X_1 \to X_0 \to M \to 0
$$
such that $X_i \in \mf(B,0)$ is proper for all $i$ and each of the induced sequences 
$$
\dots \to B(X_m) \to \cdots  \to B(X_1) \to B(X_0) \to B(M) \to 0,
$$
$$
\dots \to Z(X_m) \to \cdots \to Z(X_1) \to Z(X_0) \to Z(M) \to 0,
$$
and
$$
\dots \to \Hlf(X_m) \to \cdots \to \Hlf(X_1) \to \Hlf(X_0) \to \Hlf(M) \to 0
$$
is also exact is called a {\em Cartan-Eilenberg resolution} of $M$. Such a resolution is {\em bounded} if $X_j = 0$ for all $j \gg 0$. 
\end{defn}

%Recall $- \otimes_{\LF} -$ denotes the bifunctor from $\LF(B,0) \times \LF(B,0)$ to $\LF(B,0)$ given by 
%the tensor product of $\Z/2$-graded complexes of $B$-modules and that it corresponds to the usual notion of the tensor product of dg-$\tB$-modules.

\begin{lem} \label{lem730b} If $B$ is a Noetherian commutative ring, and at least one of $X, Y \in \mf(B,0)$ is proper, then there is a natural isomorphism
$$
\Hlf(X) \otimes_{\LF} \Hlf(Y) \xra{\cong}
\Hlf(X \otimes_{\LF} Y).
$$
\end{lem}

\begin{proof} 
The proof is the same as for the classical K\"unneth Theorem. 
\end{proof}

\begin{lem}  \label{lemCE} If $B$ is a Noetherian commutative ring, then 
every $M \in \lf(B,0)$ admits a Cartan-Eilenberg resolution. If $B$ is regular,  every $M \in \lf(B,0)$ admits a bounded Cartan-Eilenberg resolution. 
\end{lem}

\begin{proof} Choose projective resolutions of $B^0(M)$, $B^1(M)$, $\Hlf^0(M)$  and $\Hlf^1(M)$, and make repeated use of the Horseshoe Lemma, just as in the proof of the classical version of this result. If $B$ is regular, all of the chosen projective resolutions in the proof may be chosen to be bounded. 
\end{proof}

Recall that $[\mf(B,0)]$ denotes the category with the same objects as 
$\mf(B,0)$ and with morphism sets given by 
$\Hom_{[\mf(B,0)]}(X,Y) := H^0(\HomLF(X,Y))$.  
We write $\cD(\lf(B,0))$ for the category obtained from $\lf(B,0)$ 
by inverting all quasi-isomorphisms.

\begin{prop}\label{prop726a}
If $B$ is regular, the canonical functor
$$[\mf(B,0)] \xra{\cong} \cD(\lf(B,0))$$
is an equivalence.
\end{prop}

\begin{proof}
Let $M$ be an object in $\cD(\lf(B,0))$. Applying Lemma~\ref{lemCE}, choose a bounded Cartan Eilenberg resolution $X_\cdot$ of $M$. Then the canonical map $\Tot(X_\cdot) \to M$ is a quasi-isomorphism, and $\Tot(X_\cdot)$ is an object of $\mf(B,0)$; thus, the functor is essentially surjective. Fully faithfulness follows from Lemma~\ref{lem728d}.
\end{proof}

We are especially interested in complexes with finite length homology. 
Let $\lffl(B,0)$ and $\mffl(B,0)$ denote the full subcategories of $\lf(B,0)$ and $\mf(B,0)$ consisting of those complexes $M$ such that $H^0(M)$ and $H^1(M)$ are
finite length $B$-modules. Since this condition is preserved by quasi-isomorphism, we may form $[\mffl(B,0)]$ and $\cD(\lffl(B,0))$, and they may be
identified as full subcategories  of $[\mf(B,0)]$ and $\cD(\lf(B,0))$. Moreover, it follows from Proposition \ref{prop726a}
that the canonical functor induces an
equivalence
$$
[\mffl(B,0)] \xra{\cong} \cD(\lffl(B,0)),
$$
provided $B$ is regular.

It will be convenient to give an alternative description of the category $\LF(B,0)$ and of the constructions just described.
Fix a degree two indeterminate $t$ and form the $\Z$-graded algebra $\tB := B[t, t^{-1}]$, which we regard as a dg-ring with trivial differential. 
Recall that a dg-$\tB$-module is a graded $\tB$-module $M$ equipped with a degree one $\tB$-linear map $d: M \to M$ such that  $d^2 = 0$. 
Since $t$ is a degree two invertible element, a dg-$\tB$-module is the
same things as a $\Z$-graded complex of $B$-modules $M$ together with
a specified isomorphism $t: M \xra{\cong} M[2]$ of complexes.
A morphism between two such pairs, say from $(M,t)$ to $(M', t')$, is a chain map  from $M$ to $M'$ 
that commutes with $t$, $t'$. 
There is an evident equivalence of abelian categories
$$
\operatorname{dg}\text{-$\tB$-}\operatorname{Mod} \xra{\cong} \LF(B,0)
$$
that sends a dg-$\tB$-module $M$ to  the object $(M^0 \xra{d} M^1 \xra{t^{-1}d} M^0)$ of $\LF(B,0)$.  
Moreover, the notions of mapping complex, tensor product, quasi-isomorphism, homotopy equivalence and totalization 
defined above for $\LF(B,0)$ correspond to the standard
notions for  dg-modules. 
This equivalence thus allows us to employ standard results from differential graded algebra, as found, for example, in \cite{AFH}.

%%%%%%%%%%%%%%%%%%%%%%%
\subsection{Adams operations on $\Z/2$-graded complexes with finite length homology}
\label{technical}
%%%%%%%%%%%%%%%%%%%%%%%
Let $Q$ be a regular local ring with maximal ideal $\fm$. Recall that $\mf^{\fm}(Q,0)$ is the category of $\Z/2$-graded complexes of finite rank free $Q$-modules whose homology has support in $\{\fm\}$; notice that $\mf^{\fm}(Q,0) = \mffl(Q,0)$, where the right-hand side is as defined in Section~\ref{properties}.

Recall that $K_0^{\fm}(Q)$ is 
the Grothendieck group of the category of bounded $\Z$-graded complexes of projective $Q$-modules whose homology has support in $\{\fm\}$. It is easy to prove that 
$K_0^{\fm}(Q)$ is a free abelian group of rank one, generated by the class of the Koszul complex on a regular system of
generators of $\fm$. One might thus expect the answer to the following question to be positive:

\begin{quest} \label{quest726} For a regular local ring $(Q, \fm)$, is $K_0(\mf^\fm(Q,0))$ 
a free abelian group of rank one, generated by the $\Z/2$-folded Koszul complex?
\end{quest}

We know the answer to be ``yes'' if $\dm(Q) \leq 2$, but the general situation remains unknown. 
The following example illustrates the difficulty:

\begin{ex}\label{koszul} Let $(Q, \fm)$ be a regular local ring of dimension three, and suppose $x,y,z$ 
form a regular sequence of generators for the maximal ideal $\fm$. Let
$$
0 \to Q \xra{i} Q^3 \xra{A} Q^3 \xra{p} Q \to 0
$$
be the usual Koszul complex on $x,y,z$ 
(so that, for example, $p$ is given by the row matrix $(x, y, z)$). 
The $\Z/2$-folding of this Koszul complex, 
$$
K := \left(Q^3 \oplus Q \xra{\tiny{\begin{bmatrix}  A & 0 \\ 0 & 0   \end{bmatrix}}} Q^3 \oplus Q
\xra{\tiny{\begin{bmatrix}  0 & i \\ p & 0   \end{bmatrix}}} Q^3 \oplus Q \right),
$$
determines a class $[K]$ in $K_0(\mf^\fm(Q,0))$. 

Now define $B: Q^3 \to Q^3$ to be the map $i \circ p$. 
Then  $AB = 0 = BA$, so that $X = (Q^3 \xra{A} Q^3 \xra{B} Q^3)$ is a $\Z/2$-graded complex. Moreover, $\ker(B) = \im(A)$ and $\ker(A)/\im(B) \cong Q/\fm$, so that
$X \in \mf^{\fm}(Q,0)$. We do not know whether $[X]$ is a multiple of $[K]$ in $K_0(\mf^\fm(Q,0))$.
\end{ex}

To explain the relevance of Question \ref{quest726}, let us define the {\em Euler characteristic} of an object $X \in \mf^{\fm}(Q,0)$ to be
$$
\chi(X) = \len H^0(X) - \len H^1(X).
$$
Then $\chi$ determines a group homomorphism
$$
\chi: K_0(\mf^\fm(Q,0)) \to \Z.
$$
For example, if $K$ is the $\Z/2$-folded Koszul complex on a regular system of generators for $\fm$, then $\chi(K) = 1$.  
Assume now that $Q$ is a regular local $A_p$-algebra for a prime $p$ (that is, assume $p$ is invertible in $Q$ and that $Q$ contains a primitive $p$-th root of unity), so that the cyclic Adams operation $\cPsi^p$ acts on
$K_0(\mf^{\fm}(Q,0))$. We have $\cPsi^p([K]) = p^d [K]$, where $d = \dm(Q)$, by Corollary \ref{cor728}.  
If the answer to Question \ref{quest726} were
affirmative, we would obtain as an immediate consequence the identity 
\begin{equation} \label{E62}
\chi \circ \cPsi^p = p^d \chi
\end{equation}
of maps from $K_0(\mf^\fm(Q,0))$ to $\Z$.
Moreover, this equation plays a key role in the proof of Theorem \ref{introthm5}.

Although we are unable to answer Question \ref{quest726},
we are nevertheless able to prove:

\begin{thm} \label{KeyLemma} For a regular local ring $Q$ of dimension
  $d$ that is an $A_p$-algebra for some prime $p$, equation
  \eqref{E62} holds. 
\end{thm}

The proof of this theorem occupies the remainder of this subsection.

Fix a prime $p$, and let $B$ be a commutative Noetherian $A_p$-algebra.
Recall the functor $t^p_\zeta$ defined on $\mf(B,0)$ that  
sends $X$ to $T^p(X)^{(\zeta)}$, where $\zeta$ is a $p$-th root of unity. 
It will be useful to interpret this functor as a composition 
$$
\mf(B,0) \xra{T^p}  \mf(B',0) \xra{Y \mapsto Y^{(\zeta)}} \mf(B,0)
$$
where we set $B' = B[C_p] = B[\sigma]/(\sigma^p - 1)$. Since $B$ is an $A_p$-algebra, $B'$ is isomorphic to a product of $p$ copies of $B$ equipped with an
action of $C_p$. 
So, an
object of $\mf(B',0)$ is the same thing as an object of $\mf(B,0)$ equipped with an action of $C_p$, 
and  if $B$ is regular, then so is $B'$.

The functors above preserve the condition that homology has finite length, and they send homotopic maps to homotopic maps, so
that we have an induced functor

$$
t^p_\zeta: [\mffl(B,0)] \to [\mffl(B,0)]
$$
given as the composition of functors
$$
[\mffl(B,0)] \xra{T^p} [\mffl(B',0)] \xra{Y \mapsto Y^{(\zeta)}} [\mffl(B,0)].
$$

We will need a ``derived'' version of the functor $t^p_\zeta$. When $B$ is regular, then we may use the equivalence of Proposition \ref{prop726a} to obtain a functor
$$
\bt^p_\zeta: \cD(\lffl(B,0)) \to [\mffl(B,0)]. 
$$
Explicitly, for $M \in \lffl(B,0)$, $\bt^p_\zeta(M) = t^p_\zeta(P)$ where $P$ is any object of $\mffl(B,0)$ for which there exists a quasi-isomorphism $P
\xra{\sim} M$.

Given $M \in \lf(B,0)$, recall that $\Hlf(M)$ denotes the object of $\lf(B,0)$ given by the 
$\Z/2$-graded $B$-module with components  $H^0(M)$ and $H^1(M)$,
regarded as a complex with trivial differential. In terms of the dg-ring $\tB$, 
$\Hlf(M)$ corresponds to the homology of a dg-$\tB$-module, which is
naturally a dg-$\tB$-module with trivial differential (since $\tB$ has trivial differential). If $M \in \lffl(B,0)$, we define its Euler characteristic by
$$
\chi(M) := \len H^0(M) - \len H^1(M),
$$
as above.

\begin{lem}
\label{key-sub-lemma} If $B$ is a regular $A_p$-algebra, then 
for any $M \in \lffl(B,0)$ and any $p$-th root of unity $\zeta$, we have
$$
\chi(\bt^p_\zeta(M)) = \chi(\bt^p_\zeta(\Hlf(M))).
$$
\end{lem}

Theorem \ref{KeyLemma} is a relatively easy consequence of Lemma
\ref{key-sub-lemma}. Before proving Lemma \ref{key-sub-lemma}, we must introduce the following notation and establish one more preliminary result. For a bounded complex 
$$
X_\cdot = (0 \to X_m \to X_{m-1} \to \cdots \to X_1 \to X_0 \to 0)
$$
of objects of  $\LF(B,0)$, we write $\cH_q(X_\cdot) \in \LF(B,0)$ for
its homology taken in the abelian category $\LF(B,0)$; that is,
$$
\cH_q(X_\cdot) = \ker(X_q \to X_{q-1})/\im(X_{q+1} \to X_q).
$$ 
We write $\Hlf(X_\cdot)$ for the complex of objects of $\LF(B,0)$ obtained by applying $\Hlf$ term-wise:
$$
\Hlf(X_\cdot) := (0 \to \Hlf(X_d) \to \cdots \to \Hlf(X_0) \to 0).
$$
Note that $\Hlf(X_\cdot)$ is a complex of $\Z/2$-graded modules, 
and we regard it as another complex of objects in $\LF(B,0)$.

\begin{lem} \label{LemA}
For a Noetherian commutative ring $B$, assume  
$$
Y_\cdot := (0 \to Y_m \to \cdots \to Y_0 \to 0)
$$
is a complex in $\lf(B,0)$ such that both 
$\cH_q \Hlf(Y_\cdot)$ and $\Hlf \cH_q(Y_\cdot)$ have finite length for all $q$. 
Then $\Tot(Y_\cdot)$ belongs to $\lffl(B,0)$, and we have
\begin{align*}
\chi(\Tot(Y_\cdot)) &= 
\sum_{q \in \Z, \e \in \Z/2} (-1)^{q + \e} \len \cH_q(\Hlf^\e(Y_\cdot))\\
&=
\sum_{q \in \Z, \e \in \Z/2} (-1)^{q + \e} \len \Hlf^\e(\cH_q(Y_\cdot)).
\end{align*}
\end{lem}

\begin{proof}  Our proof uses spectral sequences and is
similar to the proof of the analogous fact concerning $\Z$-graded
bicomplexes, but some  care is needed to deal with the $\Z/2$-grading.

We find it most convenient to work in the setting of
dg-$\tB$-modules. Recall that a dg-$\tB$-module is the same thing
as pair consisting of a $\Z$-graded complex of $B$-modules
and a degree $2$ automorphism. A graded $\tB$-module is a 
dg-$\tB$-module with trivial differential.

Let us say that a graded $\tB$-module $H$ has {\em finite length} if $H^i$ has finite length as a $B$-module for each $i \in \Z$ (or, equivalently, for $i = 0,1$).
In this case, we define
$$
\tchi(H) = \len_B(H^0) - \len_B(H^1).
$$
(Note that $\tchi(H)  = \len_B(H^{2m}) - \len_B(H^{2n+1})$ 
for any $m, n \in \Z$.) It is clear that if $Y \in \lffl(B,0)$, then
$$
\chi(Y) =  \tchi(\tilde{H}(Y))
$$ 
where $\chi$ is as defined before, and $\tilde{H}(Y)$ denotes the homology of $Y$ regarded in the canonical
way as a graded $\tB$-module.

We will need the following fact: If $(M,d)$ is a dg-$\tB$-module
such that the underlying graded $\tB$-module $M$ has finite length, then $H(M,d)$ also has finite length, and $\tchi(H(M,d)) = \tchi(M)$. This is seen to hold by a
straightforward calculation.

We view $Y_\cdot$ as a bicomplex $Y_\cdot^\cdot$ with $m+1$ rows, whose $m$-th row, for $0 \leq j \leq m$, is
$$
\cdots \to Y_j^{-1} \to Y_j^{0} \to Y_j^{1} \to \cdots,
$$
along with  a degree $(2,0)$ isomorphism of bicomplexes $t: Y_\cdot^\cdot \xra{\cong} Y_\cdot^{\cdot+2}$.
Since this bicomplex is uniformly bounded in the vertical direction, we
have two strongly convergent spectral sequences of the form
$$
{}'E_2^{p,-q} = H_q(H^p(Y_\cdot^\cdot)) \Longrightarrow H^{p-q}(\Tot(Y_\cdot^\cdot))
$$
and
$$
{}''E_2^{p,-q} = H^p(H_q(Y_\cdot^\cdot)) \Longrightarrow H^{p-q}(\Tot(Y_\cdot^\cdot)).
$$

Let $E_r^{*,*}, r \geq 2$ refer to either of these two spectral sequences. 
The isomorphism $t: Y_\cdot^\cdot \xra{\cong} Y_\cdot^{\cdot +2}$ 
induces isomorphisms 
$$
t: E_r^{p,-q} \xra{\cong} E_r^{p+2,-q}
$$
for each $r \geq 2$, and similarly on the underlying $D_r$-terms, and these isomorphisms commute with all the maps of the exact couple. 

For any $r$, define a $\Z$-graded $B$-module $\Tot(E_r)$ by 
$$
\Tot(E_r)^n := \bigoplus_{p+q = n} E_r^{p,q}.
$$
The isomorphism $t$ induces an isomorphism of degree $2$ on $\Tot(E_r)$ making it into a 
graded $\tB$-module.
For each $r$, the differential $d_r$ on the $E_r$'s induces a degree one map 
(which we will also write as $d_r$) on $\Tot(E_r)$, and since this map 
commutes with $t$, we have that $(\Tot(E_r), d_r)$ is a dg-$\tB$-module. 
Finally, we have an identity
$$
\Tot(E_{r+1}) = H(\Tot(E_r), d_r)
$$
of graded $\tB$-modules. 

Returning to the two specific instances of this spectral sequence,
the assumptions give that each of $\Tot({}'E_2)$ and $\Tot({}''E_2)$ has finite length, and that we have
\begin{equation} \label{E61}
\begin{aligned}
\tchi(\Tot({}' E_2)) &  = \sum_{q \in \Z, \e \in \Z/2} (-1)^{q + \e} \len \cH_q(\Hlf^\e(Y_\cdot)) \\
\tchi(\Tot({}'' E_2)) & = \sum_{q \in \Z, \e \in \Z/2} (-1)^{q + \e} \len \Hlf^\e(\cH_q(Y_\cdot)). \\
\end{aligned}
\end{equation}
By the general fact mentioned above, we get that each of $\Tot(E_3), \Tot(E_4), \dots$ also has finite
length, and, moreover,
$$
\tchi(\Tot(E_2)) = \tchi(\Tot(E_3)) = \cdots = \tchi(\Tot(E_\infty)).
$$
(Note that the spectral sequence degenerates after at most 
$m+2$ steps, so that $E_{m+2} = E_{m+3} = \cdots = E_\infty$.)

Now, for $\e = 0,1$, the $B$-module $H^\e \Tot(Y)$ admits a filtration by $B$-submodules whose subquotients are 
$E_\infty^{\e,0}, E_\infty^{\e-1,1}, \dots, E_\infty^{\e-m, m}$, and hence
\begin{align*}
\chi(\Tot(Y)) &= \tchi(H(\Tot(Y)) \\  
&= \sum_q \len E_\infty^{-q,q} - 
\sum_q \len E_\infty^{1-q,q} = \tchi(\Tot(E_\infty)) = \tchi(\Tot(E_2)). 
\end{align*}
By \eqref{E61}, the proof is complete.
\end{proof}

\begin{proof}[Proof of Lemma \ref{key-sub-lemma}] We may assume,
  without loss of generality, that $M = P$ belongs to $\mffl(B,0)$. 
Let 
$$
\cdots \to 0 \to X_m \to X_{m-1} \to \cdots \to X_1 \to X_0 \to P \to 0
$$
be a bounded Cartan-Eilenberg resolution of $P$.
Since $P$ is an object of $\mf(B,0)$, the induced quasi-isomorphism
$\Tot(X_\cdot) \xra{\sim} P$ is  a homotopy equivalence, a fact
that will be used below.
 
Recall that $X_i$ is proper. In particular, $H(X_i)$ is projective for
all $i$, and the induced complex
$$
\cdots \to 0 \to H(X_m) \to H(X_{m-1}) \to \cdots \to H(X_1) \to H(X_0) \to H(P) \to 0
$$
is also exact.
The latter gives, by definition,
\begin{equation} \label{E62d}
\bt^p_\zeta(\Hlf(P))  = t^p_\zeta(\Tot(\Hlf(X_\cdot))) = T^p(\Tot(\Hlf(X_\cdot)))^{(\zeta)}.
\end{equation}

%Note that although $H(X_i)$ 
%is not necessarily of finite length, $\Tot(X_\cdot)$ belongs to
%$\mffl(B,0)$, since $\Tot(X_\cdot) \sim P$. 

For any bounded  complex $Y_\cdot$ of objects of $\mf(B,0)$, write 
$T^p(Y_\cdot)$  for the complex of objects in $\mf(B,0)$ that, in degree $j$, is
$$
T^p(Y_\cdot)_j = \bigoplus_{i_1 + \cdots + i_p = j} Y_{i_1} \otimes_{\LF} \cdots \otimes_{\LF} Y_{i_p}.
$$
For example, if $p = 2$, then $T^2(Y_\cdot)$
is the complex 
$$
\cdots \to (Y_2 \otimes Y_0  \oplus Y_1 \otimes Y_1 \oplus Y_0 \otimes Y_2)
\to (Y_1 \otimes Y_0  \oplus Y_0 \otimes Y_1) \to Y_0 \otimes Y_0 \to 0.
$$
Each term of the complex $T^p(Y)$ admits an evident signed action by $C_p$, 
and the maps of this complex respect these actions, so that we may regard 
$T^p(Y_\cdot)$ as a complex in $\mf(B',0)$, where $B' := B[C_p]$. 
We have an identity
\begin{equation} \label{E62b}
T^p(\Tot(Y_\cdot))  =  \Tot(T^p(Y_\cdot)) 
\end{equation}
of objects of $\mf(B',0)$.

Since $B$ is an $A_p$-algebra, $(-)^{(\zeta)}$ is an exact
functor from $\lf(B',0)$ to $\lf(B,0)$. In fact, $B'$ is a product of
copies of $B$, and this functor is given by extension of scalars along one of
the canonical projections $B' \onto B$. In particular, we have
\begin{equation} \label{E62e}
\Tot(Y_\cdot)^{(\zeta)} = \Tot\left(Y_\cdot^{(\zeta)}\right)
\end{equation}
for any bounded complex $Y_\cdot$ of objects of $\lf(B',0)$, and
\begin{equation} \label{E62f}
\Hlf(Y)^{(\zeta)} = \Hlf(Y^{(\zeta)})
\end{equation}
for any object $Y \in \lf(B',0)$.

Since each $X_i$ is proper, Lemma \ref{lem730b} implies that we have canonical isomorphisms 
$$
\Hlf(X_{i_1}) \otimes_{\LF} \cdots \otimes_{\LF} \Hlf(X_{i_p}) 
\xra{\cong} \Hlf(X_{i_1}  \otimes_{\LF} \cdots \otimes_{\LF} X_{i_p})
$$
which combine to give an  isomorphism
\begin{equation} \label{E62c}
T^p(\Hlf(X_\cdot)) \xra{\cong} \Hlf(T^p(X_\cdot))
\end{equation}
of complexes of objects of $\mf(B',0)$.

Combining these facts gives
$$
\begin{aligned}
\bt^p_\zeta(\Hlf(P))  & = T^p(\Tot(\Hlf(X_\cdot)))^{(\zeta)}, \text{ by \eqref{E62d},} \\
& = \left(\Tot(T^p(\Hlf(X_\cdot)))\right)^{(\zeta)}, \text{ by \eqref{E62b},} \\
& = \Tot\left(T^p(\Hlf(X_\cdot))^{(\zeta)}\right) , \text{ by \eqref{E62e},} \\     
& = \Tot\left(\Hlf(T^p(X_\cdot))^{(\zeta)}\right),    \text{ by \eqref{E62c},} \\
& = \Tot\left(\Hlf(T^p(X_\cdot)^{(\zeta)})\right) , \text{ by \eqref{E62f}.} \\ 
\end{aligned}
$$

We now apply Lemma \ref{LemA} to the complex
$Y_\cdot := T^p(X_\cdot)^{(\zeta)}$
of objects in $\mf(B,0)$, which gives
\begin{equation} \label{E61b}
\sum_{q, \e} (-1)^{q+\e} \len \cH_q(\Hlf^\e(Y_\cdot)) = 
\sum_{q,\e} (-1)^{q+\e} \len \Hlf^\e(\cH_q(Y_\cdot)).
\end{equation}
Since we have shown
$\Tot(\Hlf(Y_\cdot)) \cong \bt^p_\zeta(\Hlf(P))$, the left-hand side of \eqref{E61b} is
$\chi (\bt_\zeta^p(\Hlf(P)))$.

Recall that, since $P$ belongs to $\mf(B,0)$, the quasi-isomorphism $\Tot(X_\cdot)
\xra{\sim} P$ is a homotopy equivalence. It follows that the map
$$
\Tot(Y_\cdot)
\cong
T^p(\Tot(X_\cdot))^{(\zeta)}
\to T^p(P)^{(\zeta)}.
$$ 
is also a homotopy equivalence.
We get
$$
\Hlf^\e(\cH_q(Y_\cdot)) \cong
\begin{cases}
\Hlf^\e(t^p_\zeta(P)), & \text{if $q = 0$ and } \\
0, & \text{otherwise,} \\
\end{cases}
$$
which shows that the right-hand side of \eqref{E61b} is $\chi(t^p_\zeta(P))$.
\end{proof}

\begin{proof}[Proof of Theorem \ref{KeyLemma}]
Let $P \in \mf^\fm(Q,0) = \mffl(Q,0)$. By definition,
$$
\chi\left(\cPsi^p([P])\right) = 
\sum_{\zeta} \zeta \chi(t^p_{\zeta}(P)).
$$
By Lemma \ref{key-sub-lemma}, the value of the right-hand side of this equation 
coincides with $\sum_{\zeta} \zeta \chi(\bt^p_{\zeta}(\Hlf(P)))$.
Since $\Hlf(P)$ has trivial differential, the class 
$$
[\Hlf(P)] \in K_0(\cD(\lffl(Q,0))) \cong K_0(\mf^\fm(Q,0))
$$ 
is an integer multiple of the class of the residue field $k =
Q/\fm$, which in turn coincides with the class of 
the folded Koszul
complex $K \in \mf^\fm(Q,0)$. This proves that the equation of 
Theorem \ref{KeyLemma} holds in general provided it holds for the class $[K]$, 
and that special case is known to hold by Corollary \ref{cor728}.  
\end{proof}

%%%%%%%%%%%%%%%%%%%%%%%
\subsection{Proof of the conjecture} 
\label{sec-dao}
%%%%%%%%%%%%%%%%%%%%%%%

Throughout this section, we assume $(Q, \fm)$ is a regular local ring and $f$ is a non-zero element of $\fm$, and we set $R = Q/(f)$. We also assume $R$ is an
isolated singularity; that is, we assume $R_\fp$ is regular for all $\fp \in \Spec(R) \setminus \{\fm\}$.   
Recall from the introduction that these conditions lead to a well-defined invariant for a pair $(M,N)$ of finitely generated $R$-modules:
$$
\theta_R(M,N) = \len\left(\Tor_{2n}^R(M,N)\right) -  \len\left(\Tor_{2n+1}^R(M,N)\right)
$$
for $n \gg 0$.

For a finitely generated $R$-module $M$, $[M]_\stable$  denotes its associated class in $K_0(\mf(Q,f))$, given by the surjection $G_0(R) \onto 
K_0(\mf(Q,f))$ described in Subsection 2.3. 
Recall that $[M]_\stable = [\Fold(P,d,s)]$,
where $P$ is a $Q$-projective resolution of $M$ admitting a degree one endomophism $s$ that satisfies $ds + sd = f$ and $s^2 = 0$, that is, a Koszul resolution.

For a matrix factorization $X \in \mf(Q,f)$, 
write $X^\circ$ for $\mult_{-1}X \in \mf(Q,-f)$. That is, 
if $X = (\a:P_1 \to P_0, \b: P_0 \to P_1)$, then
$X^\circ = (\a, -\b)$.
We also use the notation $(-)^\circ$ to denote the induced isomorphism 
$K_0(\mf(Q,f)) \xra{\cong} K_0(\mf(Q,-f))$. 
For a finitely generated $R$-module $N$, the class $[N]^\circ_\stable$ is the image of $[N]$ under $G_0(R) \onto K_0(\mf(Q,-f))$, using that $Q/(f) = Q/(-f)$.

\begin{prop} \label{prop727b}
For $Q, \fm, f, R, M$ and $N$ as in Definition \ref{def728},
$$
\theta_R(M,N) = \chi \left( [M]_\stable \cup [N]_\stable^\circ \right).
$$
\end{prop}

\begin{proof}
First note that, since $f$ is an isolated singularity, one has 
$$K_0(\mf(Q,\pm f)) = K_0(\mf^{\fm}(Q,\pm f))$$ 
and hence  
$$[M]_\stable \cup [N]_\stable^\circ \in K_0(\mf^{\fm}(Q,f+(-f))) = K_0(\mf^{\fm}(Q,0)).$$

Choose matrix factorizations $X=(d_1: X_1 \to X_0, d_0: X_0 \to X_1)$ and $Y = (d_1': Y_1 \to Y_0, d_0': Y_0 \to Y_1)$ such that $[X]=[M]_\stable$ and $[Y]=[N]_\stable^{\circ}$. Assume, without loss of generality, that $N$ is maximal Cohen-Macaulay, and $N = \coker(d_1')$. 

Let $Z$ denote the object $(0 \to N, N \to 0)$ of $\lf(Q, -f)$; here, 0 is in odd degree and $N$ is in even degree. Let $\alpha: Y \to Z$ be the morphism in $\lf(Q, -f)$ given by the canonical surjection in even degree and, of course, the zero map in odd degree. Since $\theta(M,N)$ clearly coincides with the Euler characteristic of $X \otimes Z$, it suffices to show that the morphism
$$\id \otimes \alpha: X \otimes Y \to X \otimes Z$$
in $\lf(Q, 0)$ is a quasi-isomorphism. The map $\id \otimes \alpha$ is clearly surjective, so it suffices to show that its kernel is acyclic. An easy calculation shows that $\ker(\id \otimes \alpha) \cong X \otimes T$, where $T$ is the object $(Y_1 \xrightarrow{\id} Y_1, Y_1 \xrightarrow{-f} Y_1) \in \lf(Q, -f)$. Since $T$ is contractible, $X \otimes T$ is contractible; thus, $\id \otimes \alpha$ is a quasi-isomorphism.
\end{proof}

We now prove the conjecture of Dao-Kurano:

\begin{thm} \label{dao}
Let $(Q, \fm)$ be a regular local ring and  $f \in \fm$ a non-zero element, and assume $R := Q/(f)$ is an isolated singularity. 
If $M$ and $N$ are finitely generated $R$-modules such that 
$$
\dim M + \dim N \leq \dim R
$$
then $\theta_R(M,N)=0$. 
\end{thm}

\begin{proof}
Let $p$ be any prime that is invertible in $Q$. We start by reducing to the case where $Q$ contains a primitive $p$-th root of unity.
If not, we 
form the faithfully flat extension $Q \subseteq Q'$
where $Q'$ is the localization of $Q[x]/(x^p-1)$ at any one of the maximal ideals lying over $\fm$, and set $R' = Q'/f \cong R \otimes_Q Q'$. Note that $R
\subseteq R'$ is also faithfully flat, and thus
$$
\Tor_i^R(M,N) \otimes_R R' \cong \Tor_i^{R'}(M \otimes_R R', N \otimes_R R').
$$
It follows that
$$
\theta_{R'}(M \otimes_R R',N \otimes_R R')  = [R'/\fm': R/\fm] \cdot \theta_R(M, N), 
$$
and so we may replace $Q$ with $Q'$.

Set $d = \dim Q$, $c_M =\codim_Q M$ and $c_N =\codim_Q N$.  
The hypothesis that $\dim M + \dim N \leq \dim R = d-1$ yields $c_M+c_N \geq d+1$. 
By Theorem~\ref{psi-eigenspaces}, the classes $[M]_\stable, [N]_\stable \in K_0(\mf(Q,f)) \otimes \Q$ decompose uniquely as
$$
[M]_\stable  = \sum_{i=c_M}^{d}  X_i
$$
and 
$$
[N]_\stable = \sum_{j=c_N}^{d}  Y_j
$$
where $X_i, Y_j$ are such that $\cPsi^p(X_i) = p^i X_i$ and 
$\cPsi^p(Y_j) = p^j Y_j$.  
Then 
$$
[N]_\stable^\circ = \sum_{j=c_N}^{d}  Y_j^\circ
$$
and, 
by Corollary \ref{cor727c}, $\cPsi^p(Y_j^\circ) = p^j Y_j^\circ$ for all $j$. 

By Proposition \ref{prop727b}, we have
$$
\theta_R(M, N) = \chi([M]_\stable \cup [N]_\stable^\circ) = \sum_{i,j} \chi(X_i \cup Y_j^\circ),
$$
and so it suffices to prove $\chi(X_i \cup Y_j^\circ) = 0$ for all $i,j$. 
For any $i,j$, 
\begin{align*}
p^d \chi(X_i \cup Y_j^\circ)   
&= \chi(\cPsi^p(X_i \cup Y_j^\circ)) \\
&= \chi(\cPsi^p(X_i) \cup \cPsi^p(Y_j^\circ)) \\
&= \chi(p^i X_i \cup p^j Y_j^\circ) \\
&= p^{i+j} \chi(X_i \cup Y_j^\circ),  
\end{align*} 
where the first equality is by Theorem \ref{KeyLemma},
the second is by Theorem \ref{psi-GS-axioms-mf},
and the third is by definition of $X_i$ and $Y_j$. 
Since Theorem \ref{psi-eigenspaces} yields that 
$i+j \geq c_M + c_N  > d$, we conclude that 
$\chi(X_i \cup Y_j^\circ)  = 0$.
\end{proof}

%%%%%%%%%%%%%%%%%%%%%%%
%%%%%%%%%%%%%%%%%%%%%%%
\bibliographystyle{amsalpha}
\bibliography{bibliography}
%%%%%%%%%%%%%%%%%%%%%%%
%%%%%%%%%%%%%%%%%%%%%%%

\end{document}